\documentclass[a4paper,12pt]{amsart}

\usepackage[english]{babel}
\usepackage{amsmath,amssymb,amsfonts,amsthm,enumerate}
\usepackage{url}
\usepackage{graphicx,epstopdf,color}
\usepackage[colorlinks=true, allcolors=blue]{hyperref}
\usepackage{amsrefs}
\usepackage{mathtools}
\usepackage[shortlabels]{enumitem}

\numberwithin{equation}{section}
\allowdisplaybreaks

\newtheorem{theorem}{Theorem}[section]
\newtheorem{lemma}[theorem]{Lemma}

\newtheorem{proposition}[theorem]{Proposition}
\newtheorem{corollary}[theorem]{Corollary}

\newcommand{\R}{\mathbb{R}}

\newcommand{\N}{\mathbb{N}}

\renewcommand{\epsilon}{\varepsilon}
\newcommand{\eps}{\varepsilon}
\newcommand{\F}{\mathcal{F}}

\newcommand{\U}{\mathcal{U}}

\renewcommand{\leq}{\leqslant}
\renewcommand{\le}{\leqslant}
\renewcommand{\geq}{\geqslant}
\renewcommand{\ge}{\geqslant}

\newcommand{\dist}{\operatorname{dist}}

\allowdisplaybreaks

\title[Asymptotic development of nonlocal phase transition energies]
{A threshold for higher-order asymptotic development\\
of genuinely
nonlocal phase transition energies}

\author[S.~Dipierro]{\href{https://research-repository.uwa.edu.au/en/persons/serena-dipierro}{Serena Dipierro} }

\author[E.~Valdinoci]{\href{https://research-repository.uwa.edu.au/en/persons/enrico-valdinoci}{Enrico Valdinoci} }

\author[M. Vaughan]{\href{https://maryvaughan.github.io/}{Mary Vaughan}}

\address[S. Dipierro, E.~Valdinoci, M.~Vaughan]{The University of Western Australia, 
Department of Mathematics and Statistics,
35 Stirling HWY,
Crawley WA 6009, Australia}
\email{serena.dipierro@uwa.edu.au}
\email{enrico.valdinoci@uwa.edu.au}
\email{mary.vaughan@uwa.edu.au}

\keywords{Gamma-convergence, 
nonlocal phase transitions, 
fractional Allen--Cahn equation}

\subjclass[2010]{
82B26, 
35R11, 
49J45 
}

\begin{document}

\maketitle

\begin{abstract}
We study the higher-order asymptotic development of a nonlocal phase transition energy in bounded domains and with prescribed external boundary conditions.  
The energy under consideration has fractional order~$2s \in (0,1)$ and 
a first-order asymptotic development in the~$\Gamma$-sense as described by the fractional perimeter functional.

We prove that there is no meaningful second-order asymptotic expansion and, in fact, no asymptotic expansion of fractional order~$\mu > 2-2s$. 
In view of this range value for~$\mu$,
it would be interesting to develop a new asymptotic development for the~$\Gamma$-convergence of our energy functional which
takes into account fractional orders.

The results obtained here are also valid in every space dimension and with mild
assumptions on the exterior data.
\end{abstract}

\section{Introduction}
In this note, we study the higher-order asymptotic development of a long-range phase coexistence energy functional in the genuinely nonlocal regime under prescribed external data. The mathematical setting under consideration is that of an energy of the form
\begin{equation}\label{eq:Feps}
\mathcal{F}_\eps(u) := \eps \iint_{(\R^n \times \R^n) \setminus (\Omega^c \times \Omega^c)} \frac{|u(x) - u(y)|^2}{|x-y|^{n+2s}} \, dx\,dy
+\eps^{1-2s}\int_{\Omega} W(u(x)) \, dx,
\end{equation}
where~$\Omega \subset \R^n$
is a bounded {Lipschitz} domain, the notation~$\Omega^c$ stands, as usual, for~$\R^n\setminus\Omega$, $W$ is a double-well potential with minima at~$-1$ and~$1$, and~$\eps>0$ is a small parameter.

The fractional parameter range considered here is~$s \in \left(0,\frac12\right)$, which corresponds to significant far-away interactions which preserve the nonlocal character of the model at every scale: in particular, in this framework and, as shown in~\cite{SV-gamma}, the minimizers of~$\eps^{-1}\mathcal{F}_\eps$ 
approach, as~$\eps\searrow0$, a step function whose jump set is a minimizer for the fractional perimeter introduced in~\cite{CRS}.

We consider as ambient space~$X_g$ for our analysis that of measurable functions~$u:\R^n\to[-1,1]$ subject to the external datum prescription~$u=g$ in~$\R^n\setminus\Omega$ for a given continuous function~$g:\R^n \to [-1,1]$. 

It has been recently proved in~\cite[Theorem~1.1]{PAPERONE} that, as~$\eps\searrow0$, the minimal value~$m_\eps$ attained by the energy functional~\eqref{eq:Feps}
in~$X_g$ takes the form~$m_\eps=\eps m_1+o(\eps)$ with
\[
m_1 := \inf_{\substack{u~\text{s.t.} \\ u |_{\Omega} = \chi_E - \chi_{E^c}}} \left[ \displaystyle\iint_{\Omega\times\Omega}\frac{|u(x)-u(y)|^2}{|x-y|^{n+2s}}\,dx\,dy
+2\displaystyle\iint_{\Omega\times\Omega^c}\frac{|u(x)-g(y)|^2}{|x-y|^{n+2s}}\,dx\,dy\right].
\]
Here, the infimum is taken over measurable sets~$E \subset \R^n$ and, as customary,
$$ \chi_E(x):=\begin{dcases}1&{\mbox{ if }}x\in E,\\0&{\mbox{ otherwise.}}\end{dcases}$$

It is therefore desirable to understand the terms with higher orders in~$\eps$ which contribute to the minimal energy. This question can be addressed via the version of the~$\Gamma$-convergence theory (with respect to the metric in~$L^1(\Omega)$) put forth in~\cite{Baldo1, Baldo2}. In this framework, the asymptotic development of order~$k \in \N$ is written as
$$ \F_\eps =_{\Gamma} \F^{(0)} +\eps \F^{(1)} +\dots+\eps^k \F^{(k)}+o(\eps^{k})$$
and corresponds to the following conditions:
\begin{enumerate}
\item $\F^{(0)} =
\Gamma-\displaystyle\lim_{\eps\searrow0} \F_\eps$ in~${\mathcal{U}}_{-1}$;
\item for any~$j\in \{0,\dots, k-1\}$, we have that
$\F^{(j+1)} =\Gamma-\displaystyle\lim_{\eps\searrow0} \F_\eps^{(j+1)}$ 
in~${\mathcal{U}}_j$, where
\begin{eqnarray*}
&& {\mathcal{U}}_j := \{ u\in 
{\mathcal{U}}_{j-1} {\mbox{ s.t. }} \F^{(j)}(u)=m_j\}\\[.5em]
{\mbox{with }}&& m_j :=\inf_{{\mathcal{U}}_{j-1}}\F^{(j)}\\
{\mbox{and }}&& \F^{(j+1)}_\eps:=\frac{\F^{(j)}_\eps-m_j}{\eps}
.\end{eqnarray*}
\end{enumerate}
The above iteration starts with~${\mathcal{U}}_{-1}$ being the ambient functional space and~$
\F^{(0)}_\eps := \F_\eps$.

In this setting, one can show (see~\cite[pages~106 and 110]{Baldo1}) that 
\[
\{\text{limits of minimizers of}~\mathcal{F}_\eps\} \subset \U_k \subset \dots \subset \U_0 \subset \U_{-1}
\]
and  
\begin{equation*}
m_\eps = m_0 + \eps m_1 + \dots + \eps^k m_k+o(\eps^{k}).
\end{equation*}

The first-order $\Gamma$-convergence expansion of \eqref{eq:Feps} established in~\cite[Theorem~1.5]{PAPERONE} is described by
\begin{equation*}
\F^{(0)}(u) = 0 \quad \hbox{for all}~u \in X_g\end{equation*}
and
\begin{equation*}
\F^{(1)}(u)
= \begin{cases}
\displaystyle\iint_{\Omega \times \Omega} \frac{|u(x) - u(y)|^2}{|x-y|^{n+2s}}\,dx\,dy +2\iint_{\Omega\times\Omega^c}
\frac{|u(x)-g(y)|^2}{|x-y|^{n+2s}}\,dx \,dy & 
\begin{matrix}
{\mbox{ if~$u \in X_g$ and}}\\
{\mbox{$u\big|_\Omega = \chi_E - \chi_{E^c}$}} \\ {\mbox{ for some $E \subset \R^n$,}}\end{matrix} \\
\\
+\infty & {\mbox{ otherwise.}}
\end{cases}
\end{equation*}

When~$n=1$, $\Omega = (-1,1)$, and~$g=\chi_{(0,+\infty)}-\chi_{(-\infty,0)}$, it was also shown in~\cite[Corollary~1.10]{PAPERONE} that
\begin{equation*}
\F^{(2)}(u)
= \begin{dcases}
-\infty & \hbox{if}~u~\hbox{is a minimizer of}~\F^{(1)} \quad \hbox{(i.e.~$u \in \mathcal{U}_1$)},\\
+\infty & \hbox{otherwise}.
\end{dcases}
\end{equation*}
We show here that the same expansion holds true in every dimension~$n\ge1$ and for external data~$g$ whose values are not necessarily confined in~$\{-1,1\}$. More precisely, we have the following:

\begin{theorem}\label{MSA:TH}
Suppose~$m_1 \not=0$ and let~$\bar{u} \in \mathcal{U}_1$ be such that~$\bar{u} \big|_{\Omega} = \chi_E - \chi_{E^c}$ for some measurable set~$E \subset \R^n$. Assume at least one of the following holds:
\begin{enumerate}[label={(I.\arabic*)}]
\item \label{item:interior}
there exist~$x_0 \in \partial E$ and~$r>0$ such that~$B_{r}(x_0) \Subset \Omega$ and~$\partial E \cap B_r(x_0)$ is a~$C^{1,\alpha}$ surface;
\item \label{item:boundary}
there exists a connected component~$\Omega'$ of~$\Omega$ such that neither~$g \equiv 1$ nor~$g \equiv -1$ on~$\partial \Omega'$. 
\end{enumerate} 

Then, there exists a sequence~$v_\eps \in X_g$ such that~$v_\eps \to \bar{u}$ as~$\eps \searrow 0$ and, for any~$\mu>1-2s$,  
\[
\lim_{ \eps \searrow 0} \frac{\mathcal{F}_\eps^{(1)}(v_\eps) - m_1}{\eps^{\mu}} = - \infty.
\]
In particular,
\[
\lim_{ \eps \searrow 0} \mathcal{F}_\eps^{(2)}(v_\eps) = - \infty.
\]
\end{theorem}

It follows that there is no nontrivial second-order asymptotic development of~$\mathcal{F}_\eps$, and in fact, any meaningful higher-order asymptotic development must be of \emph{fractional} order.

\begin{corollary}\label{cor:F1}
In the setting of Theorem~\ref{MSA:TH}, we have that
\[
\mathcal{F}^{(2)}(u)
 = \begin{cases}
-\infty & \hbox{if $u$ is a minimizer of $\mathcal{F}^{(1)}$ \quad \hbox{(i.e.~$u \in \mathcal{U}_1$)}},\\
+ \infty & \hbox{otherwise}.
\end{cases}
\]
\end{corollary}

We remark that condition~\ref{item:interior} corresponds to the minimizer~$\bar{u}$ having a jump discontinuity in the interior and condition~\ref{item:boundary} corresponds to~$\bar{u}$ having a jump discontinuity at the boundary. 
Indeed, suppose~$\bar u \in \mathcal{U}_1$ is such that~$\bar{u}\big|_{\Omega} = \chi_E - \chi_{E^c}$ for some~$E \subset \R^n$. 
We will prove in Proposition~\ref{ALM} that~$E$ is an almost minimizer to the fractional perimeter functional, and hence~\ref{item:interior} holds for regular points~$x_0 \in \partial E \cap \Omega$, see~\cite{DVV}. 
On the other hand, if~\ref{item:interior} does not hold, then either~$\bar{u} \equiv 1$ or~$\bar{u} \equiv -1$ and jumps discontinuities at the boundary are guaranteed by~\ref{item:boundary}. 

The sequence in Theorem~\ref{MSA:TH} is constructed by lifting/lowering the minimizer~$\bar{u}$ on one side of a jump discontinuity  in small balls whose radii are infinitesimal as~$\eps \searrow 0$ (see Proposition~\ref{prop:anydim65}). We then iterate this process by applying a simple covering lemma to construct a sequence~$v_\eps$ whose energy is bounded above by~$-\kappa \eps^{1-2s}$ for some constant~$\kappa>0$.  

Note that, when~$m_1=0$, we obtain the full asymptotic development,
according to the following statement:

\begin{theorem}\label{lem:trivial}
If~$m_1=0$, then, for all~$k \geq 2$, 
\[
\mathcal{F}^{(k)}(u) = \begin{cases}
0 & \hbox{if $u \in \U_1$},\\
+ \infty & \hbox{otherwise},
\end{cases}
\]
and, consequently, $m_k = 0$ and~$\U_k = \U_1$. 
\end{theorem}

Following~\cite{Baldo1, Baldo2}, there have been several works devoted to higher-order asymptotic development in the local setting, see for instance~\cite{Braides} for additional details. 
We reference the reader to~\cite{MR759767, MR2971613, MR3385247} for a higher-order~$\Gamma$-convergence expansion of a classical, one-dimensional, Modica-Mortola energy functional and~\cite{DalMaso, Leoni} for the higher-dimensional case. 
Nevertheless, our results are unique to the nonlocal framework and do not follow from the classical theory. 

\subsection{Notation}

For convenience, we write the domain of integration in~\eqref{eq:Feps} as
\[
Q_\Omega := (\R^n \times \R^n) \setminus (\Omega^c \times \Omega^c)= (\Omega \times \Omega) \cup (\Omega^c \times \Omega) \cup (\Omega \times \Omega^c),
\]
and denote the kinetic energy of a measurable function~$u:\R^n \to \R$ by
\[
u(Q_\Omega) := \iint_{Q_\Omega} \frac{|u(x) - u(y)|^2}{|x-y|^{n+2s}} \, dx \, dy. 
\]

\subsection{Paper organization}

The rest of the paper is organized as follows. 
First, in Section~\ref{sec:almost}, we prove that minimizers of~$\mathcal{F}^{(1)}$ correspond to almost minimizers of the fractional perimeter functional. 
Then we quantify the drop in kinetic energy allowed by lifting/lowering the minimizer on one side of the jump discontinuity in Section~\ref{sec:energy}. 
An elementary covering lemma is given in Section~\ref{sec:covering}. 
The proof of Theorem~\ref{MSA:TH} is contained in Section~\ref{sec:main}
and the proof of Theorem~\ref{lem:trivial} in Section~\ref{appendix:2}.

For the sake of completeness, Appendix~\ref{appe:n1} contains a more precise quantification in the drop in kinetic energy in dimension~$1$. 

\section{Almost minimizers}\label{sec:almost}

In this section, given a measurable set~$E \subseteq \R^n$,  we consider the functional 
\begin{equation}\label{yrueifhndjgfdh54yu3}\begin{split}
{\mathcal{H}}(E,\Omega)&:=
\iint_{\Omega\times\Omega}\frac{|(\chi_{E}(x)- \chi_{E^c}(x))-(\chi_{E}(y)- \chi_{E^c}(y))|^2}{|x-y|^{n+2s}}\,dx\,dy\\
&\qquad+2\iint_{\Omega\times(\R^n\setminus\Omega)}\frac{|(\chi_{E}(x)- \chi_{E^c}(x))-g(y)|^2}{|x-y|^{n+2s}}\,dx\,dy.
\end{split}\end{equation}
Notice that if~$u \in X_g$ satisfies~$u|_{\Omega} = \chi_E - \chi_{E^c}$, then~$u(Q_\Omega) = \mathcal{H}(E,\Omega)$.

Recall from~\cite{CRS} that the fractional perimeter of a set~$E$ in~$\Omega$ is given by
\[
\operatorname{Per}_s(E,\Omega) 
=\int_{E\cap\Omega}\int_{E^c}\frac{dx\,dy}{|x-y|^{n+2s}}+
\int_{E \cap \Omega^c}\int_{E^c \cap \Omega}\frac{dx\,dy}{|x-y|^{n+2s}}. 
\]

Our goal is to show that minimizers of the functional~$\mathcal{H}(\cdot, \Omega)$
are almost minimizers of the fractional perimeter in a smaller domain. 

\begin{proposition}\label{ALM} Let~$\Omega'\Subset\Omega$ and suppose that~$E$ is a minimizer for~${\mathcal{H}}(\cdot,\Omega)$.

Then, there exists~$\Lambda >0$, depending only on~$n$, $s$, $\Omega$, and~$\Omega'$, 
such that, for all~$F\subseteq\R^n$
such that~$F\setminus\Omega'=E\setminus\Omega'$,
$$\operatorname{Per}_s(E,\Omega')\le\operatorname{Per}_s(F,\Omega')+\Lambda|E\Delta F|.$$
Moreover, $\Lambda = C[\dist(\partial \Omega', \partial \Omega)]^{-2s}$,
for some~$C>0$ depending only on~$n$ and~$s$.
\end{proposition}

\begin{proof} If~$F$ is as above, we see that 
\begin{eqnarray*}
&&\operatorname{Per}_s(F,\Omega')\\&&\quad=\int_{F\cap\Omega'}\int_{F^c\cap\Omega'}\frac{dx\,dy}{|x-y|^{n+2s}}+
\int_{F\cap\Omega'}\int_{F^c\setminus\Omega'}\frac{dx\,dy}{|x-y|^{n+2s}}+
\int_{F\setminus\Omega'}\int_{F^c\cap\Omega'}\frac{dx\,dy}{|x-y|^{n+2s}}\\
&&\quad=\frac12\iint_{\Omega'\times\Omega'}\frac{|\chi_F(x)-\chi_F(y)|^2}{|x-y|^{n+2s}}\,dx\,dy+
\iint_{\Omega'\times(\R^n\setminus\Omega')}\frac{|\chi_F(x)-\chi_F(y)|^2}{|x-y|^{n+2s}}\,dx\,dy\\
&&\quad=\frac12\iint_{\Omega\times\Omega}\frac{|\chi_F(x)-\chi_F(y)|^2}{|x-y|^{n+2s}}\,dx\,dy
-\iint_{\Omega'\times(\Omega\setminus\Omega')}\frac{|\chi_F(x)-\chi_F(y)|^2}{|x-y|^{n+2s}}\,dx\,dy\\&&\qquad\quad
-\frac12\iint_{(\Omega\setminus\Omega')\times(\Omega\setminus\Omega')}\frac{|\chi_F(x)-\chi_F(y)|^2}{|x-y|^{n+2s}}\,dx\,dy\\&&\qquad\quad
+
\iint_{\Omega'\times(\R^n\setminus\Omega')}\frac{|\chi_F(x)-\chi_F(y)|^2}{|x-y|^{n+2s}}\,dx\,dy\\
&&\quad=\frac12\iint_{\Omega\times\Omega}\frac{|\chi_F(x)-\chi_F(y)|^2}{|x-y|^{n+2s}}\,dx\,dy
-\frac12\iint_{(\Omega\setminus\Omega')\times(\Omega\setminus\Omega')}\frac{|\chi_F(x)-\chi_F(y)|^2}{|x-y|^{n+2s}}\,dx\,dy\\&&\qquad\quad
+
\iint_{\Omega'\times(\R^n\setminus\Omega)}\frac{{|\chi_F(x) - \chi_{F}(y)|^2}}{|x-y|^{n+2s}}\,dx\,dy.
\end{eqnarray*}
Notice that
\[
|(\chi_E(x)- \chi_{E^c}(x)) - (\chi_E(y)- \chi_{E^c}(y))|
	= 2 |\chi_E(x) - \chi_E(y)|.
\]
Therefore,
\begin{eqnarray*}
&&\operatorname{Per}_s(E,\Omega')-\operatorname{Per}_s(F,\Omega')
\\
&&\quad=\frac12\iint_{\Omega\times\Omega}\frac{|\chi_E(x)-\chi_E(y)|^2}{|x-y|^{n+2s}}\,dx\,dy
-\frac12\iint_{\Omega\times\Omega}\frac{|\chi_F(x)-\chi_F(y)|^2}{|x-y|^{n+2s}}\,dx\,dy\\&&\qquad\quad
+
\iint_{\Omega'\times(\R^n\setminus\Omega)}\frac{|\chi_E(x) -\chi_E(y)|^2-|\chi_F(x) -\chi_F(y)|^2}{|x-y|^{n+2s}}\,dx\,dy 
\\
&&\quad=
\frac18\iint_{\Omega\times\Omega}\frac{|(\chi_E(x)- \chi_{E^c}(x)) - (\chi_E(y)- \chi_{E^c}(y))|^2}{|x-y|^{n+2s}}\,dx\,dy
\\&&\qquad\quad-\frac18\iint_{\Omega\times\Omega}\frac{|(\chi_F(x)- \chi_{F^c}(x)) - (\chi_F(y)- \chi_{F^c}(y))|^2}{|x-y|^{n+2s}}\,dx\,dy\\&&\qquad\quad
+
\iint_{\Omega'\times(\R^n\setminus\Omega)}\frac{|\chi_E(x) -\chi_E(y)|^2-|\chi_F(x) -\chi_F(y)|^2}{|x-y|^{n+2s}}\,dx\,dy 
\\
&&\quad=\frac18\big({\mathcal{H}}(E,\Omega)-{\mathcal{H}}(F,\Omega)\big)\\&&\qquad\quad
-\frac14\iint_{\Omega\times(\R^n\setminus\Omega)}\frac{|\chi_E(x) - \chi_{E^c}(x)-g(y)|^2}{|x-y|^{n+2s}}\,dx\,dy\\&&\qquad \quad
+
\frac14\iint_{\Omega\times(\R^n\setminus\Omega)}\frac{|\chi_F(x) - \chi_{F^c}(x)-g(y)|^2}{|x-y|^{n+2s}}\,dx\,dy\\&&\qquad \quad
+
\iint_{\Omega'\times(\R^n\setminus\Omega)}\frac{|\chi_E(x) -\chi_E(y)|^2-|\chi_F(x) -\chi_F(y)|^2}{|x-y|^{n+2s}}\,dx\,dy.
\end{eqnarray*}
Thus, the minimality of~$E$ for the functional~${\mathcal{H}}(\cdot,\Omega)$
and the fact that~$F\setminus\Omega'=E\setminus\Omega'$ give that
\begin{eqnarray*}
&&\operatorname{Per}_s(E,\Omega')-\operatorname{Per}_s(F,\Omega')
\\&&\quad\le -
\frac14\iint_{\Omega' \times (\R^n \setminus \Omega)}\frac{
|\chi_E(x) - \chi_{E^c}(x)-g(y)|^2-|\chi_F(x) - \chi_{F^c}(x)-g(y)|^2}{|x-y|^{n+2s}}\,dx\,dy\\
&&\qquad\quad
+
\iint_{\Omega'\times(\R^n\setminus\Omega)}\frac{|\chi_E(x) -\chi_E(y)|^2-|\chi_F(x) -\chi_E(y)|^2}{|x-y|^{n+2s}}\,dx\,dy .
\end{eqnarray*}

Notice also that
\[
|(\chi_E(x)- \chi_{E^c}(x)) - (\chi_F(x)- \chi_{F^c}(x))|
 	= 2|\chi_E(x) - \chi_F(x)|.
\]
Hence, for~$x \in \Omega'$ and~$y \in \R^n \setminus \Omega$, 
\begin{align*}
&\left|
|\chi_E(x) - \chi_{E^c}(x)-g(y)|^2-|\chi_F(x) - \chi_{F^c}(x)-g(y)|^2
\right|\\
&\quad =\big| (\chi_E(x) - \chi_{E^c}(x)) -  (\chi_F(x) - \chi_{F^c}(x))\big|\;
	\big| (\chi_E(x) - \chi_{E^c}(x)) +  (\chi_F(x) - \chi_{F^c}(x)) - 2g(y)\big|\\
&\quad\leq 8| \chi_E(x) -\chi_F(x)|
\end{align*}
and
\begin{align*}
\big||\chi_E(x) -\chi_E(y)|^2-|\chi_F(x) -\chi_E(y)|^2\big|
	&= |\chi_E(x) - \chi_F(x)|\;|\chi_E(x) + \chi_F(x)-2\chi_E(y)|\\
	&\leq 4|\chi_E(x) - \chi_F(x)|.
\end{align*}
Using this, we estimate
\begin{eqnarray*}
&&\operatorname{Per}_s(E,\Omega')-\operatorname{Per}_s(F,\Omega')
\\&&\leq 
\frac14\iint_{\Omega' \times (\R^n \setminus \Omega)}\frac{
\left||\chi_E(x) - \chi_{E^c}(x)-g(y)|^2-|\chi_F(x) - \chi_{F^c}(x)-g(y)|^2\right|}{|x-y|^{n+2s}}\,dx\,dy\\
&&\qquad\quad
+
\iint_{\Omega'\times(\R^n\setminus\Omega)}\frac{\big||\chi_E(x) -\chi_E(y)|^2-|\chi_F(x) -\chi_E(y)|^2\big|}{|x-y|^{n+2s}}\,dx\,dy\\
&&\quad\le 6
\iint_{\Omega'\times(\R^n\setminus\Omega)}\frac{|\chi_E(x)-\chi_F(x)|}{|x-y|^{n+2s}}\,dx\,dy\\&&\quad\le  6
\iint_{\Omega'\times(\R^n\setminus B_\rho)}\frac{|\chi_E(x)-\chi_F(x)|}{|z|^{n+2s}}\,dx\,dz 
\\&&\quad=  C |E \Delta F| \rho^{-2s}, 
\end{eqnarray*}
where~$\rho>0$ denotes the distance between~$\partial\Omega$ and~$\partial\Omega'$, and the desired result follows.
\end{proof}

\section{Energy estimates}\label{sec:energy}

We now quantify the energy drop allowed by a jump discontinuity.

\begin{proposition}\label{prop:anydim65}
Let~$n\ge1$, $b>0$, and~$c\in\left(0,\frac12\right)$.
Let~$\Omega\subset\R^n$ and~$\bar u:\R^n\to[-1,1]$.

Then, there exist~$\varsigma>0$ and~${\theta \in (0,b)}$, depending only on~$n$, $s$, $b$ and~$c$,
such that the following statement holds true.

Assume that there exist~$\delta>0$ and~$p$, $q\in\R^n$ such that~$|p-q|<\delta$, $B_{c\delta}(p)\subseteq\Omega$, $\bar u=-1$ in~$B_{c\delta}(p)$, and~$\bar u\ge-1+b$ in~$B_{c\delta}(q)$.

Then, the function 
\[
u_\delta(x) = \begin{cases} \bar{u}(x) & \hbox{if}~x \in \R^n \setminus B_{c \delta}(p) ,\\
-1+ \theta & \hbox{if}~x \in B_{c \delta}(p)
\end{cases}
\]
satisfies
$$ u_\delta(Q_\Omega)\le\bar u(Q_\Omega)- \varsigma \delta^{n-2s}.$$
\end{proposition}

\begin{proof} If~$u=\bar u$ in~$\R^n\setminus B_{c\delta}(p)$, then
\begin{eqnarray*}&&
u(Q_\Omega)-\bar u(Q_\Omega)=
\iint_{Q_\Omega}\frac{|u(x)-u(y)|^2-|\bar u(x)-\bar u(y)|^2}{|x-y|^{n+2s}}\,dx\,dy\\
&&\qquad=\iint_{Q_\Omega}\frac{\big[ \big(u(x)-\bar u(x)\big)-\big(u(y)-\bar u(y)\big)\big]\big[ \big(u(x)+\bar u(x)\big)-\big(u(y)+\bar u(y)\big)\big]}{|x-y|^{n+2s}}\,dx\,dy\\
&&\qquad=\iint_{B_{c\delta}(p)\times B_{c\delta}(p)}\frac{\big[ \big(u(x)-\bar u(x)\big)-\big(u(y)-\bar u(y)\big)\big]\big[ \big(u(x)+\bar u(x)\big)-\big(u(y)+\bar u(y)\big)\big]}{|x-y|^{n+2s}}\,dx\,dy\\
&&\qquad\qquad+2\iint_{{B_{c\delta}(p)\times(\R^n\setminus B_{c\delta}(p))}}\frac{\big(u(x)-\bar u(x)\big)\big(u(x)+\bar u(x)-2\bar u(y)\big)}{|x-y|^{n+2s}}\,dx\,dy.
\end{eqnarray*}

In particular, given~$\theta\in(0,{b})$, to be taken conveniently small in what follows, if~$u=-1+\theta=\bar u+\theta$ in~$B_{c\delta}(p)$, we evince that
\begin{equation}\label{qsdjoclnDqosch0I}\begin{split}&u(Q_\Omega)-\bar u(Q_\Omega)=
2\theta\iint_{{B_{c\delta}(p)\times(\R^n\setminus B_{c\delta}(p))}}\frac{\theta-2-2\bar u(y)}{|x-y|^{n+2s}}\,dx\,dy.
\end{split}\end{equation}

We also observe that~$\theta-2-2\bar u\le\theta$ and accordingly
\begin{eqnarray*}&&\iint_{{B_{c\delta}(p)\times\big(\R^n\setminus (B_{c\delta}(p)\cup B_{c\delta}(q))\big)}}\frac{\theta-2-2\bar u(y)}{|x-y|^{n+2s}}\,dx\,dy\\&&\qquad\le
\iint_{{B_{c\delta}(p)\times(\R^n\setminus B_{c\delta}(p))}}\frac{\theta}{|x-y|^{n+2s}}\,dx\,dy\\&& \qquad=\theta \,\operatorname{Per}_s(B_{c\delta}(p),\R^n)\\&& \qquad=C_0\,\theta \delta^{n-2s},
\end{eqnarray*}
for some~$C_0>0$, depending only on~$n$ and~$s$.

This, in tandem with~\eqref{qsdjoclnDqosch0I}, leads to
\begin{eqnarray*}&&\frac{u(Q_\Omega)-\bar u(Q_\Omega)}{2\theta}
\le\iint_{{B_{c\delta}(p)\times B_{c\delta}(q)}}\frac{\theta-2-2\bar u(y)}{|x-y|^{n+2s}}\,dx\,dy+C_0\,\theta \delta^{n-2s}\\&&\qquad\le
\iint_{{B_{c\delta}(p)\times B_{c\delta}(q)}}\frac{\theta-2b}{|x-y|^{n+2s}}\,dx\,dy+C_0\,\theta \delta^{n-2s}.
\end{eqnarray*}

Also, if~$x\in B_{c\delta}(p)$ and~$y\in B_{c\delta}(q)$, we have that~$|x-y|\le|p-q|+2c\delta\le 2\delta$ and accordingly
$$ \iint_{{B_{c\delta}(p)\times B_{c\delta}(q)}}\frac{dx\,dy}{|x-y|^{n+2s}}\ge
\iint_{{B_{c\delta}(p)\times B_{c\delta}(q)}}\frac{dx\,dy}{(2\delta)^{n+2s}}=
C_1\,c^{2n}\delta^{n-2s},$$
for some~$C_1>0$, depending only on~$n$ and~$s$.

As a result, taking also $$\theta\le\frac{2C_1\,b\,c^{2n}}{2C_0+C_1\,c^{2n}},$$ we find that
\begin{eqnarray*}&&\frac{u(Q_\Omega)-\bar u(Q_\Omega)}{2\theta}
\le -C_1\,c^{2n}\,(2b-\theta)\delta^{n-2s}+C_0\,\theta \delta^{n-2s}\le
-\frac{C_1\,c^{2n}\,(2b-\theta)}2\,\delta^{n-2s},
\end{eqnarray*}yielding the desired result.

\end{proof}

When~$n=1$, the result in Proposition~\ref{prop:anydim65} can be made more precise.
Though we will not use this refined calculation in this paper,
for the sake of completeness we present it in Appendix~\ref{appe:n1}.

\section{Covering lemma}\label{sec:covering}

To iterate Proposition~\ref{prop:anydim65}, we need the following simple covering lemma. The proof is straightforward, but we write the details for completeness. 

\begin{lemma}\label{lem:covering}
Let~$E \subset \R^n$ be a measurable set with nontrivial boundary.
Suppose that there exist~$x_0 \in \partial E$ and~$r>0$ such that~$\partial E \cap B_r(x_0)$ is a Lipschitz graph with Lipschitz constant~$L>0$.

Then, there exist constants~$C_1$, $c_1>0$, depending only on~$n$, $r$ and~$L$, such that, for~$\delta>0$ sufficiently small,
there exist~$N_\delta \in \N$ satisfying
\[
c_1 \delta^{1-n} \leq N_\delta \leq C_1\delta^{1-n}
\]
and a finite sequence of points~$x_i \in\partial E \cap B_r(x_0)$ 
with~$i\in\{1,\dots,N_\delta\}$ such that 
\[
B_\delta(x_i) \subset B_r(x_0) \qquad  \hbox{and} \qquad |x_i-x_j|>2\delta
\quad \hbox{for}~i\not=j. 
\] 
\end{lemma}

\begin{proof}
Without loss of generality, assume that~$x_0=0$. 
We denote~$x = (x',x_n) \in \R^{n-1}\times \R$. 
Assume, up to a rotation, that~$\psi = \psi(x')$ is a Lipschitz function with constant~$L$ such that~$\psi(0) = 0$ and 
\[
\partial E \cap B_r  = \big\{(x', \psi(x'))~\hbox{s.t.}~x' \in U\big\}
\]
for some bounded, open Lipschitz set~$U \subset \R^{n-1}$. 

Denote a ball in~$\R^{n-1}$ of radius~$t>0$ and centered at~$x' \in \R^{n-1}$ by~$B_t'(x')$ and set~$B_t' := B_t'(0)$. 
Fix~$\rho>0$ such that~$(1+L)\rho < r$ and~$B_\rho' \subset U$. 
 
Then, there exist constants~$C_n$, $c_n>0$, depending only on~$n$, such that, for~$\delta>0$ sufficiently small, there exist~$N_\delta \in \N$ and points~$x_i' \in B_\rho'$
with~$i\in\{1,\dots,N_\delta\}$
such that
\[
c_n \left( \frac{\rho}{(1+L)\delta}\right)^{n-1} \leq N_\delta \leq C_n \left( \frac{\rho}{(1+L)\delta}\right)^{n-1}
\]
and
\[
B_{(1+L) \delta}'(x_i') \subset B_{\rho}' \quad \hbox{and} \quad
|x_i'-x_j'|>2(1+L)\delta \quad \hbox{for}~i\not=j.
\]

Now define~$x_i := (x_i', \psi(x_i'))$ and note that~$x_i \in \partial E \cap B_r$. 
Moreover, we have that~$B_\delta(x_i) \subset B_r$. Indeed, 
if~$y \in B_\delta(x_i)$, then
\begin{align*}
&|y| \leq |y-x_i| + |(x_i',\psi(x_i')|
	\leq \delta + |x_i'| + |\psi(x_i') -\psi(0)|\\
	&\qquad\qquad\leq \delta + (1+L)  |x_i'|
	\leq \delta + (1+L) \rho < r,
\end{align*}
for $\delta$ sufficiently small, and so~$y\in B_r$, as desired.

Lastly, for~$i\not=j$, we see that
\[
|x_i - x_j|
	\geq |x'_i - x'_j|-|\psi(x_i')-\psi(x_j')| 
	\geq |x'_i - x'_j| - L |x_i' - x_j'|
> 2(1+L)\delta -L|x_i -x_j|
\]
so that~$|x_i - x_j|>2 \delta$. 
\end{proof}

\section{Proof of Theorem~\ref{MSA:TH}}\label{sec:main}

\begin{proof}[Proof of Theorem~\ref{MSA:TH}]
First of all, we point out that when dealing with sets, up to a modification
of null measure, we implicitly choose
a ``good'' representative with the ``smallest possible boundary'',
see~\cite[Appendix~A]{MR3516886} for the details.

Now, since~$\bar{u} \in\mathcal{U}_1$, we have that~$E$ is a minimizer
for the functional~${\mathcal{H}}(\cdot,\Omega)$ in~\eqref{yrueifhndjgfdh54yu3}.
Therefore, we can exploit 
Proposition~\ref{ALM} and obtain that~$E$ is an almost minimizer of the fractional perimeter functional in compact subsets of~$\Omega$.

We now distinguish two cases:
\begin{itemize}
\item[(i)] either there exists a connected component~$\Omega'$ of~$\Omega$
such that~$\partial E\cap\Omega'\neq\varnothing$,
\item[(ii)] or for all connected components~$\Omega'$ of~$\Omega$ it holds that~$\partial E\cap \Omega'=\varnothing$.
\end{itemize}

\medskip

\underline{\bf Case~1}. 
Suppose first that~(i) holds.
In this case, 
thanks to the regularity result for almost minimizers of the fractional perimeter functional in~\cite[Theorem~1.6]{DVV},
we have that the assumption in~\ref{item:interior} is satisfied.
Therefore,
there exist~$x_0 \in \partial E$ and~$r>0$ such that~$B_{r}(x_0) \Subset \Omega$ and~$\partial E \cap B_r(x_0)$ is the graph of a~$C^{1,\alpha}$ function.  
For~$\delta>0$ sufficiently small, let~$B_{\delta}(x_i)$ with~$x_i \in \partial E \cap B_r(x_0)$, for~$i=1,\dots, N_\delta$, be as in Lemma~\ref{lem:covering}.

Since~$E$ is an almost minimizer of the fractional perimeter functional,
we are in the position of using the clean ball condition in~\cite[Corollary~2.5]{DVV}. In this way, for any~$i \in \{1,\dots, N_{\delta}\}$,
there exist points~$p_i \in \Omega \setminus E$ and~$q_i \in \Omega \cap E$ such that~$B_{c\delta}(p_i) \subset E^c \cap B_{\delta/2}(x_i)$ 
and~$B_{c \delta}(q_i) \subset E \cap B_{\delta/2}(x_i)$,
where~$c\in (0,\frac12)$ depends only on~$n$ and~$s$. 
In particular, 
\[
|p_i-q_i| < \delta, \qquad \bar{u} = -1~\hbox{in}~B_{c \delta}(p_i),\qquad \hbox{and} \qquad \bar{u}= 1~\hbox{in}~B_{c \delta}(q_i). 
\]

Hence, by Proposition~\ref{prop:anydim65}
(used here with~$b=2$), we have that the function
\begin{equation}\label{eq:u1}
u_1(x) := \begin{cases} \bar{u}(x) & \hbox{if}~x \in \R^n \setminus B_{c\delta}(p_1),\\
-1+\theta & \hbox{if}~x \in B_{c\delta}(p_1)
\end{cases}
\end{equation}
 satisfies
\[
u_1(Q_\Omega) \leq \bar{u}(Q_\Omega) - \varsigma \delta^{n-2s} 
\]
where~$\varsigma>0$ and~$\theta \in (0,2)$ depend only on~$n$, $s$, and~$c$.
 
Iterating for all~$i =2,\dots, N_{\delta}$, we apply Proposition~\ref{prop:anydim65} with~$u_{i-1}$ in place of~$\bar{u}$ to find a function
\[
u_{i}(x) := \begin{cases} u_{i-1}(x) & \hbox{if}~x \in \R^n \setminus B_{c\delta}(p_i)\\
-1+\theta & \hbox{if}~x \in B_{c\delta}(p_i),
\end{cases}
\]
which satisfies
\begin{equation}\label{eq:interated-energy}
u_i(Q_\Omega) 
	\leq u_{i-1}(Q_\Omega) - \varsigma \delta^{n-2s}.
\end{equation}

We define~$u_\delta := u_{N_\delta}$. In this way,
\[
u_\delta =
\begin{cases}
\bar{u} & \hbox{in}~\displaystyle\R^n \setminus \bigcup_{i=1}^{N_\delta} B_{c\delta}(p_i),\\
-1+\theta & \hbox{in}~\displaystyle\bigcup_{i=1}^{N_\delta} B_{c\delta}(p_i).
\end{cases} 
\]
By~\eqref{eq:interated-energy},
\begin{equation}\label{frhejdwghfedjtyrhe}
u_\delta(Q_\Omega) 
	\leq \bar{u}(Q_\Omega) -N_\delta \varsigma \delta^{n-2s}.
\end{equation}
Moreover, 
\begin{equation}\label{frhejdwghfedjtyrhe2}\begin{split}
\int_{\Omega} W(u_\delta(x)) \, dx
	&= \sum_{i=1}^{N_\delta} \int_{B_{c\delta}(p_i)} W(u_\delta(x)) \, dx
	= \sum_{i=1}^{N_\delta}  W(\theta-1)|B_{c\delta}(p_i)|
	= \omega N_\delta \delta^n 
\end{split}\end{equation}
where~$\omega>0$ depends only on~$n$, $s$, $c$, $W$, $\Omega$, and~$r$.
 
Since~$\bar{u} \in\mathcal{U}_1$, we have that~$\bar{u}(Q_\Omega) =  m_1$ and thus, by~\eqref{frhejdwghfedjtyrhe} and~\eqref{frhejdwghfedjtyrhe2}, we gather that
\begin{align*}
u_\delta(Q_\Omega) - m_1 + \frac{1}{\eps^{2s}} \int_{\Omega} W(u_\delta(x)) \, dx
	&\leq -N_\delta {\varsigma}\delta^{n-2s} +\frac{\omega N_\delta \delta^n}{\eps^{2s}}\\
	&= \left( -{\varsigma} \delta^{1-2s} +\frac{\omega \delta }{\eps^{2s}} \right)N_\delta \delta^{n-1}.
\end{align*}

By~\cite[Corollary~4.7]{PAPERONE}, the quantity 
\[
-{\varsigma} \delta^{1-2s} +\frac{ \omega \delta}{\eps^{2s}} 
\]
attains the minimum at
\begin{equation}\label{dvse5432fdxfcshgfkre5y43}
\delta = \left(\frac{(1-2s) {\varsigma}}{\omega}\right)^{\frac{1}{2s}} \eps
\end{equation} and the value at the minimum is
$$-2s \left(\frac{1-2s}{\omega}\right)^{\frac{1-2s}{2s}} {\varsigma}^{\frac{1}{2s}} \eps^{1-2s}.$$ 
Therefore,
$$ u_\delta(Q_\Omega) - m_1 + \frac{1}{\eps^{2s}} \int_{\Omega} W(u_\delta(x)) \, dx
\le -2s \left(\frac{1-2s}{\omega}\right)^{\frac{n-2s}{2s}} \bar{\varsigma}^{\frac{1}{2s}} \eps^{1-2s} N_\delta \delta^{n-1}. $$

Recall from Lemma~\ref{lem:covering} that~$N_\delta \delta^{n-1} \geq c_1>0$. 
Therefore, taking~$v_\eps := u_\delta$ with~$\delta$
as in~\eqref{dvse5432fdxfcshgfkre5y43}, we find that 
\[
v_\eps(Q_\Omega) - m_1+ \frac{1}{\eps^{2s}} \int_{\Omega} W(v_\eps(x)) \, dx
	\leq -2s c_1 \left(\frac{1-2s}{\omega}\right)^{\frac{n-2s}{2s}} \bar{\varsigma}^{\frac{1}{2s}} \eps^{1-2s}. 
\]
The desired result plainly follows. 

\bigskip

\underline{\bf Case~2}. 
Suppose now that~(ii) holds.
Therefore, in every connected component of~$\Omega$, we have that
either~$\bar{u} \equiv 1$ or~$\bar{u} \equiv -1$.
Accordingly, the assumption in~\ref{item:interior} does not hold and
we assume instead that~\ref{item:boundary} is in force.

Let~$\Omega'$ be the connected component of~$\Omega$ given by~\ref{item:boundary} and, without loss of generality, assume that~$\bar{u} \equiv -1$ in~$\Omega'$. 

By~\ref{item:boundary}, we also have that
there exists a point~$x_0 \in \partial \Omega'\subset\partial\Omega$ such that~$g(x_0) \not= -1$. 
By the continuity of~$g$, 
there exist~$r$, $b>0$ such that~$g \geq -1+b$ in~$\Omega^c \cap B_r(x_0)$. 
Possibly making~$r$ smaller, we have that~$\partial \Omega' \cap B_r(x_0)$ is a Lipschitz graph. 

Hence, we are in the position of using Lemma~\ref{lem:covering} (applied here
with~$E:=\Omega'$). Thus, for~$\delta>0$ sufficiently small, let~$B_{\delta}(x_i)$ with~$x_i \in \partial\Omega' \cap B_r(x_0)$, for~$i=1,\dots, N_\delta$, be as in Lemma~\ref{lem:covering}.

Since~$\Omega$ is Lipschitz, there exists~$c>0$, depending only on~$n$ and the Lipschitz constant of~$\partial \Omega' \cap B_r(x_0)$, such that, for all~$\delta>0$ sufficiently small and for all~$i\in\{1,\dots,N_\delta\}$, there exists points~$p_i \in \Omega'$ and~$q_i \in \Omega^c$ such that~$B_{c\delta}(p_i) \subset \Omega' \cap B_{\delta/2}(x_i)$ 
and~$B_{c \delta}(q_i) \subset \Omega^c \cap B_{\delta/2}(x_i)$. 
In particular, 
\[
|p_i-q_i| < \delta, \qquad \bar{u} = -1~\hbox{in}~B_{c \delta}(p_i),\qquad \hbox{and} \qquad \bar{u} = g >-1+b~\hbox{in}~B_{c \delta}(q_i). 
\]

The rest of the proof follows along the same lines as Case~1, starting from~\eqref{eq:u1}. 
\end{proof}

\section{Proof of Theorem \ref{lem:trivial}}\label{appendix:2}

In this section, we give the details of the proof of Theorem~\ref{lem:trivial}.

\begin{proof}[Proof of Theorem~\ref{lem:trivial}]
Since $m_1 = 0$, it is enough to consider the sequence of energies
\begin{equation}\label{eq:Feps-k}
\mathcal{F}_\eps^{(k)}(u) = \frac{1}{\eps^k} \left[ u(Q_\Omega) + \frac{1}{\eps^{2s}} \int_{\Omega} W(u(x)) \, dx \right], \quad {\mbox{ with }}k \geq 2,
\end{equation}
and show that~
\begin{equation}\label{eq:claim}
\mathcal{F}^{(k)}~\hbox{is the $\Gamma$-limit of}~\mathcal{F}_\eps^{(k)} \quad \hbox{for all}~k \geq 2. 
\end{equation}
Indeed, is it clear that
\[
\mathcal{F}_\eps^{(2)}(u) = \frac{\mathcal{F}_\eps^{(1)}(u) - 0}{\eps} = \frac{1}{\eps} \left[ u(Q_\Omega) + \frac{1}{\eps^{2s}} \int_{\Omega} W(u(x)) \, dx \right].
\]
If \eqref{eq:claim} holds with~$k=2$, then~$m_2 = m_1 = 0$ and~$\U_2 = \U_1$. 
Proceeding by induction, if, for some~$k \geq 2$, the energy~$\mathcal{F}_\eps^{(k)}$ given in~\eqref{eq:Feps-k} satisfies~\eqref{eq:claim}, then~$m_k = m_1 = 0$, $\U_k = \U_1$, and 
\[
\mathcal{F}_\eps^{(k+1)}(u) = \frac{\mathcal{F}_\eps^{(k)}(u)-0}{\eps}
=  \frac{1}{\eps^{k+1}} \left[ u(Q_\Omega) + \frac{1}{\eps^{2s}} \int_{\Omega} W(u(x)) \, dx \right].
\] 
Hence, we will now focus on the proof of~\eqref{eq:claim}.

To prove \eqref{eq:claim}, fix~$k \geq 2$ and let~$\bar u\in X_g$. If~$u_j\in X_g$ is such that~$u_j\to\bar u$ in~$X_g$ and~$\eps_j\searrow0$ as~$j\to+\infty$,
we claim that
\begin{equation}\label{yrfueibcdnsmedhw23456789}
\liminf_{j \to +\infty} \mathcal{F}_{\eps_j}^{(k)}(u_{j})  \geq  \mathcal{F}^{(k)}(\bar{u}).
\end{equation}
Indeed, if~$\bar{u} \in \mathcal{U}_1$, then trivially we have
\[
\liminf_{j \to +\infty} \mathcal{F}_{\eps_j}^{(k)}(u_{j})  \geq 0 = \mathcal{F}^{(k)}(\bar{u}), 
\]
and if~$\bar{u} \in X_g \setminus \mathcal{U}_1$, then by~\cite[Lemma~4.2]{PAPERONE}, 
\[
\liminf_{j \to +\infty} \mathcal{F}_{j}^{(k)}(u_{\eps_j}) 
	 = +\infty. 
\]
These considerations give~\eqref{yrfueibcdnsmedhw23456789}.

Furthermore, we claim that there exists a sequence~$u_{\eps_j}\in X_g$
such that~$u_{\eps_j}\to \bar u$ as~$j\to+\infty$ and
\begin{equation}\label{ghjfkedy446u3rdhsxsawqexrctfgyh}
\limsup_{j \to +\infty} \mathcal{F}_{\eps_j}^{(k)}(u_{\eps_j}) = \mathcal{F}^{(k)}(\bar{u}). 
\end{equation}
Indeed, we take~$u_{\eps_j} := \bar{u}$.
 
If~$\bar{u} \in X_g \setminus \mathcal{U}_1$, 
then
$$  \mathcal{F}^{(k)}(\bar{u})=+\infty$$
and
$$ 
\bar u(Q_\Omega) + \frac{1}{\eps^{2s}} \int_{\Omega} W(\bar u(x)) \, dx>m_1=0.$$
Therefore,
$$ \limsup_{j \to +\infty} \mathcal{F}_{\eps_j}^{(k)}(u_{\eps_j})
=\limsup_{j \to +\infty} \mathcal{F}_{\eps_j}^{(k)}(\bar u)=+\infty,$$
and~\eqref{ghjfkedy446u3rdhsxsawqexrctfgyh} holds true.

If instead~$\bar{u} \in \mathcal{U}_1$, then 
\[
	\mathcal{F}_{\eps_j}^{(k)}(u_j) = 
	\frac{m_1}{\eps_j^k}  = 0 = \mathcal{F}^{(k)}(\bar{u}),
\]
and so
\[
\liminf_{j \to +\infty} 
	\mathcal{F}_{\eps_j}^{(k)}(u_j)= \mathcal{F}^{(k)}(\bar{u}), 
\]
which gives~\eqref{ghjfkedy446u3rdhsxsawqexrctfgyh}.

{F}rom~\eqref{yrfueibcdnsmedhw23456789} and~\eqref{ghjfkedy446u3rdhsxsawqexrctfgyh}, we obtain the desired result.
\end{proof}

\begin{appendix}

\section{Refinement of Proposition~\ref{prop:anydim65} in dimension~$1$}
\label{appe:n1}

In dimension~$1$, the analysis provided by Proposition~\ref{prop:anydim65}
can be refined, since the regularity of the almost minimizers of the fractional perimeter
allows for the analysis of a sharp transition, along the lines that we discuss here below.

\begin{proposition}
Let~$p\in[0,1]$ and~$a\in(0,1)$.

Let~$\bar u:\R\to[-1,1]$ be such that~$\bar u=-1$ in~$(p-a,p)$.
Assume also that~$\bar u\ge-1+b$ in~$(p,p+a)$, for some~$b>0$.

Then, there exist~$\varsigma>0$ and~$\delta_0\in\left(0,\frac{a}2\right)$, depending only on~$s$, $a$, and~$b$, and, for all~$\delta\in(0,\delta_0)$, a function~$u_\delta:\R\to[-1,1]$ such that~$u_\delta=\bar u$ in~$\R\setminus(p-\delta,p)$ and
$$ u_\delta(Q_{(-1,1)})\le\bar u(Q_{(-1,1)})- \varsigma \delta^{1-2s}.$$
\end{proposition}

\begin{proof} Note that~$(p-\delta,p)\subseteq(-1,1)$ and therefore, if~$u=\bar u$ in~$\R\setminus(p-\delta,p)$, then
\begin{eqnarray*}&&
u(Q_{(-1,1)})-\bar u(Q_{(-1,1)})=
\iint_{Q_{(-1,1)}}\frac{|u(x)-u(y)|^2-|\bar u(x)-\bar u(y)|^2}{|x-y|^{1+2s}}\,dx\,dy\\
&&\qquad=\iint_{Q_{(-1,1)}}\frac{\big[ \big(u(x)-\bar u(x)\big)-\big(u(y)-\bar u(y)\big)\big]\big[ \big(u(x)+\bar u(x)\big)-\big(u(y)+\bar u(y)\big)\big]}{|x-y|^{1+2s}}\,dx\,dy\\
&&\qquad=\iint_{(p-\delta,p)^2}\frac{\big[ \big(u(x)-\bar u(x)\big)-\big(u(y)-\bar u(y)\big)\big]\big[ \big(u(x)+\bar u(x)\big)-\big(u(y)+\bar u(y)\big)\big]}{|x-y|^{1+2s}}\,dx\,dy\\
&&\qquad\qquad+2\iint_{{(p-\delta,p)\times(\R\setminus(p-\delta,p))}}\frac{\big(u(x)-\bar u(x)\big)\big(u(x)+\bar u(x)-2\bar u(y)\big)}{|x-y|^{1+2s}}\,dx\,dy.
\end{eqnarray*}

In particular, given~$\theta\in(0,2)$, if~$u=-1+\theta=\bar u+\theta$ in~$(p-\delta,p)$,
\begin{equation}\label{OJSLD.002k40}\begin{split}&
u(Q_{(-1,1)})-\bar u(Q_{(-1,1)})=2\theta\iint_{{(p-\delta,p)\times(\R\setminus(p-\delta,p))}}\frac{
\big(\theta-2-2\bar u(y)\big)}{|x-y|^{1+2s}}\,dx\,dy\\
&\quad=\frac\theta{s}\Bigg[\int_{-\infty}^{p-\delta}
\big(\theta-2-2\bar u(y)\big)\left( \frac1{(p-\delta-y)^{2s}}-\frac1{(p-y)^{2s}}\right)\,dy\\&\qquad\quad
+\int_p^{+\infty}\big(\theta-2-2\bar u(y)\big)\left( \frac1{(y-p)^{2s}}-\frac1{(y-p+\delta)^{2s}}\right)\,dy\Bigg].
\end{split}\end{equation}

Now we denote by~$C>0$ a positive constant depending only on~$s$, possibly varying from line to line, and observe that, for all~$\eta\in[0,1]$,
we have that~$(1+\eta)^{2s}\le1+C\eta$, and therefore
\begin{equation}\label{OJSLD.002k4}\begin{split}&
\int_{-\infty}^{p-a}\left| \frac1{(p-\delta-y)^{2s}}-\frac1{(p-y)^{2s}}\right|\,dy=
\int_0^{+\infty}\left| \frac1{(a-\delta+t)^{2s}}-\frac1{(a+t)^{2s}}\right|\,dt\\&\qquad=
\int_0^{+\infty}\frac1{(a+t)^{2s}}\left| 1-\left(\frac{a+t}{a-\delta+t}\right)^{2s}\right|\,dt\\&\qquad=
\int_0^{+\infty}\frac1{(a+t)^{2s}}\left[ \left(1+\frac{\delta}{a-\delta+t}\right)^{2s}-1\right]\,dt\\&\qquad\le
C\delta\int_0^{+\infty}\frac{dt}{(a+t)^{2s}(a-\delta+t)}\\&\qquad\le C\delta\int_0^{+\infty}\frac{dt}{(a+t)^{1+2s}}\\&\qquad\le\frac{C\delta}{a^{2s}}.
\end{split}\end{equation}

Similarly,
\begin{equation*}\begin{split}&
\int_{p+a}^{+\infty}\left| \frac1{(y-p)^{2s}}-\frac1{(y-p+\delta)^{2s}}\right|\,dy=
\int_{0}^{+\infty}\left| \frac1{(a+t)^{2s}}-\frac1{(a+\delta+t)^{2s}}\right|\,dy\le\frac{C\delta}{a^{2s}}.
\end{split}\end{equation*}

Plugging this estimate and~\eqref{OJSLD.002k4} into~\eqref{OJSLD.002k40} we find that
\begin{equation}\label{OJSLD.002k7} \begin{split}&
u(Q_{(-1,1)})-\bar u(Q_{(-1,1)})\le\frac\theta{s}\Bigg[\int_{p-a}^{p-\delta}
\big(\theta-2-2\bar u(y)\big)\left( \frac1{(p-\delta-y)^{2s}}-\frac1{(p-y)^{2s}}\right)\,dy\\&\qquad\quad
+\int_p^{p+a}\big(\theta-2-2\bar u(y)\big)\left( \frac1{(y-p)^{2s}}-\frac1{(y-p+\delta)^{2s}}\right)\,dy\Bigg]+\frac{C\theta\delta}{a^{2s}}.
\end{split}\end{equation}

We also remark that
\begin{equation}\label{OJSLD.002k7b}\begin{split}&
\int_{p-a}^{p-\delta}
\big(\theta-2-2\bar u(y)\big)\left( \frac1{(p-\delta-y)^{2s}}-\frac1{(p-y)^{2s}}\right)\,dy\\&\qquad=
\theta\int_{p-a}^{p-\delta}\left( \frac1{(p-\delta-y)^{2s}}-\frac1{(p-y)^{2s}}\right)\,dy\\&\qquad=\frac{\theta\big((a - \delta)^{1 - 2s}-
a^{1 - 2s}+ \delta^{1 - 2s} \big)}{1-2s}\\&\qquad\le\frac{\theta\,\delta^{1 - 2s} }{1-2s}.\end{split}\end{equation}
In addition, in~$(p,p+a)$ we have that~$-\theta+2+2\bar u\ge 2b-\theta>0$ as long as~$\theta\in(0,2b)$ and consequently
\begin{eqnarray*}&&
\int_p^{p+a}\big(-\theta+2+2\bar u(y)\big)\left( \frac1{(y-p)^{2s}}-\frac1{(y-p+\delta)^{2s}}\right)\,dy\\&&\qquad\ge
(2b-\theta)\int_p^{p+a}\left( \frac1{(y-p)^{2s}}-\frac1{(y-p+\delta)^{2s}}\right)\,dy
\\&&\qquad=\frac{2b-\theta}{1-2s}\left( a^{1-2s}-(a+\delta)^{1-2s}+\delta^{1-2s}\right)\\
\\&&\qquad\ge\frac{2b-\theta}{1-2s}\left( \delta^{1-2s}-\frac{C\delta}{a^{2s}}\right).
\end{eqnarray*}

We insert this information and~\eqref{OJSLD.002k7b} into~\eqref{OJSLD.002k7} and we obtain that
\begin{equation*} \begin{split}&
u(Q_{(-1,1)})-\bar u(Q_{(-1,1)})\le\frac\theta{s}\Bigg[\frac{\theta\,\delta^{1 - 2s} }{1-2s}
-\frac{2b-\theta}{1-2s}\left( \delta^{1-2s}-\frac{C\delta}{a^{2s}}\right)
\Bigg]+\frac{C\theta\delta}{a^{2s}}\\&\qquad
\le\frac{\theta\delta^{1-2s}}{s(1-2s)}\big[ \theta
-(2b-\theta)
\big]+\frac{C\theta(1+b)\delta}{a^{2s}}\\&\qquad
=\frac{2\theta(\theta-b)\delta^{1-2s}}{s(1-2s)}+\frac{C\theta(1+b)\delta}{a^{2s}}.
\end{split}\end{equation*}
The desired result now follows by choosing~$u_\delta:=u$ with~$\theta:=\frac{b}2$. 
\end{proof}

\end{appendix}

\section*{Acknowledgments}

This work has been supported by the Australian Future Fellowship
FT230100333 ``New perspectives on nonlocal equations''
and by the Australian Laureate Fellowship FL190100081 ``Minimal surfaces, free boundaries and partial differential equations.''

\bibliographystyle{imsart-number}
\bibliography{gamma}

\end{document}